\documentclass[a4paper,12pt]{amsart}

\usepackage{epsfig}
\usepackage{amsthm,amsfonts}
\usepackage{amssymb,graphicx,color}
\usepackage[all]{xy}

\newtheorem{theorem}{Theorem}[section]
\newtheorem{lemma}[theorem]{Lemma}

\newtheorem{definition}[theorem]{Definition}

\newtheorem{remark}[theorem]{Remark}

\newcommand{\R}{\mathbb R}

\begin{document}

\title[ Arc Criterion of Normal Embedding 
 ]{Arc Criterion of Normal Embedding 
}

\author{L. Birbrair} \thanks{The first author were partially supported by CAPES-COFECUB and by CNPq-Brazil, grants no. 302655/2014-0}

\address{Departamento of Matem\'atica, 
Universidade Federal do Cear\'a, Campus do Pici Bloco 914, CEP 60455--760, Fortaleza CE, Brazil}

\email{lev.birbrair@gmail.com}

\author{Rodrigo Mendes}
\address{Rodrigo Mendes - Departamento de Matem\'atica, Universidade de Integra\c{c}\~ao Internacional da Lusofonia Afro-Brasileira (unilab)
, Campus dos Palmares, Cep. 62785-000. Acarape-Ce,
Brasil} 

\email{rodrigomendes@unilab.edu.br}


\thanks{The second author has been partially supported by capes}

\begin{abstract}
We present a criterion of local normal embedding of a semialgebraic (or definable in a polynomially bounded o-minimal structure) germ contained in $\R^n$ in terms of orders of contact of arcs. Namely, we prove that a semialgebraic germ is normally embedded if and only if for any pair of arcs, coming to this point the inner order of contact is equal to the outer order of contact.
\end{abstract}

\maketitle

\section{Introduction}
A connected subset of Euclidean space is called normally embedded if the two natural metrics, outer and the inner metric, are bi-Lipschitz equivalent and the bi-Lipschitz homeomorphism is given by the identity map. The notion of Normal Embedding (or in other words Lipschitz Normal Embedding) became rather popular in recent development of Singularity Theory. Is is used in Metric Homology Theory of Birbrair and Brasselet \cite{BB}, in Vanishing Homology of Valette \cite{V}, in Lipschitz Regularity theorem \cite{BFLS}. Several authors are investigating some special algebraic and semialgebraic sets in the spirit of their normal embedding. See, for example, the recent works \cite{BFN}, \cite{NPP}, \cite{KKKLS}. In this note we present an arc criterion of Normal Embedding that, we hope, it can be useful in these studies. The criterion is based on the arc selection lemma, an extremely important tool of Real Algebraic Geometry.

\bigskip

\noindent{\bf Acknowledgements}. We would like to thank Alexandre Fernandes, Edson Sampaio, Anne Pichon and Walter Neumann for useful discussions. We would like also to thank the anonymous referee for his patience and extremely useful comments and corrections. 

\bigskip

\section{Normally embedded sets}

Let $X \subset \mathbb{R}^n$ be a connected semialgebraic set. We define an inner metric on $X$ as follows: Let $x,y \in X$. The inner distance $d_X(x,y)$ is defined as the infimum of lengths of rectifiable arcs $\gamma:[0,1]\rightarrow X$ such that $\gamma(0)=x$ and $\gamma(1)=y$. Notice that for connected semialgebraic sets the inner metric is well-defined.
\begin{definition}
 \emph{A semialgebraic subset $X\subset\R^n$ is called \emph{normally embedded} if there exists $\lambda >0$ such that}
$$d_X(x_1,x_2)\le \lambda \|x_1-x_2\|,$$ 
for all $x_1,x_2\in X$. 
\end{definition}
$X$ is called normally embedded at $x_0$ if for sufficiently small $\epsilon >0$, $X \cap B_{x_0,\epsilon}$ is normally embedded. We may also say that the germ of $X$ at $x_0$ is normally embedded. 

Considering real or complex cusps $x^2=y^3$, one can see that the inner metric is not bi-Lipschitz equivalent to the Euclidean metric. On the other hand, the smooth compact semialgebraic sets are normally embedded.

By the results of Kurdyka and Orro \cite{K} (see also Birbrair and Mostowski \cite{BM}), there exists a semialgebraic metric 
\begin{center}
$d_P: X \times X \rightarrow \R,$ 
\end{center}
such that $(X,d_X)$ and $(X,d_P)$ are bi-Lipschitz equivalent and the identity map is bi-Lipschitz for $d_X$ and $d_P$.

An {\emph{arc} $\gamma$ with initial point at $x_0$ is a continuous semialgebraic map $\gamma:[0,\epsilon)\rightarrow \mathbb{R}^n$ such that $\gamma(0)=x_0$. When it does not lead to a confusion, we use the same notation for an arc and its image in $\mathbb{R}^n$.

For a semialgebraic function of one variable $f(t)$, where $f(0)=0$, we have $f(t)=a_1t^{\alpha}+o(t^{\alpha})$, for some $a_1 \in \mathbb{R}$ and $\alpha \in \mathbb{Q}$. The number $\alpha$ is called the order of $f$ at $0$. We use the notation $ord_t f$. \\
We can define the \emph{outer order of tangency} in the following way: 
\begin{center}
$tord(\gamma_1,\gamma_2)=ord_t \|\gamma_1(t)-\gamma_2(t)\|$,
\end{center}
where the arcs $\gamma_1$ and $\gamma_2$ are parametrized by the outer distance to the singular point, i.e., $\|\gamma_i(t)-x_0\|=t$, $i=1,2$. We may define the \emph{inner order of tangency} by
\begin{center}
$tord_{inn}(\gamma_1,\gamma_2)=ord_t (d_P(\gamma_1(t),\gamma_2(t)))$,
\end{center}
 where the arcs $\gamma_1$ and $\gamma_2$ are again parametrized by the outer distance to the singular point.

\begin{theorem} (Criterion of Normal embedding) \label{criterion} \emph{Let $X\subset \R^n$ be a closed semialgebraic germ at $x_0$. Then the following assertions are equivalent:}
\begin{itemize}
\item \emph{The germ of $X$ at $x_0$ is normally embedded};
 
\item \emph{There exists a constant $k>0$ such that for any pair of arcs $\gamma_1, \gamma_2$ parametrized by the distance at $x_0$, ($\gamma_i(0)=x_0$) we have 
\begin{center}
$d_X(\gamma_1(t), \gamma_2(t))\leq k\|\gamma_1(t)-\gamma_2(t))\|$;
\end{center}
\item For any pair of arcs $\gamma_1, \gamma_2$ parametrized by the distance to $x_0$ one has:}
\begin{center}
 $tord(\gamma_1,\gamma_2)=tord_{inn}(\gamma_1,\gamma_2)$.
 \end{center}
\end{itemize}
\end{theorem}
\begin{remark}
\emph{The theorem is formulated in the semialgebraic category, but the result is true for polynomially bounded o-minimal structures. Actually, all the ingredients of the proof work in that case}. 
\end{remark}
\begin{proof}If $X$ is normally embedded at $x_0$ the inequality above follows from the definition. Assume now that $X$ is not normally embedded at $x_0$. Consider a map $\psi:X \times X \rightarrow \mathbb{R}^2$, defined as follows: $\psi(x_1,x_2)=(\|x_1-x_2\|,d_P(x_1,x_2))$. This map is semialgebraic. Since the distance functions $\|x_1-x_2\|$ and $d_P(x_1,x_2)$ are continuous, then the image $\psi(X \times X)$ is closed and locally connected at $\psi(x_0,x_0)=0 \in \mathbb{R}^2$, $x_0 \in X$. Moreover, this set is semialgebraic, according to Tarski-Seidenberg theorem. Since $X$ is not normally embedded at $x_0$, the set $\psi(X \times X)$ must be locally bounded about $0 \in \mathbb{R}^2$ by an arc $\beta \subset \psi(X \times X)$ such that $\beta$ is tangent to the $y$-axis. Taking $(\tilde{\gamma}_1(t), \tilde{\gamma}_2(t))$ belonging to the inverse image of $\beta$, we obtain that
\begin{equation}\label{order condition}
ord_t d_P(\tilde{\gamma}_1(t),\tilde{\gamma}_2(t)) < ord_t(\|\tilde{\gamma}_1(t)-\tilde{\gamma}_2(t)\|).
\end{equation}
We may suppose that the arc $\tilde{\gamma}_1$ is parametrized by the distance to the singular point $x_0$. But, we cannot suppose that the other arc is also parametrized the same way. That is why we need the order of comparison lemma.

\begin{remark}(Non-archimedean property) (see for example \cite{BF}). \emph{Let $\gamma_1, \gamma_2$ and $\gamma_3$ be three different semialgebraic arcs of $X$, $\gamma_i(0)=x_0 \ (i=1,2,3)$. Let $\alpha_{12},\alpha_{23}$ and $\alpha_{13}$ be outer orders of tangency between $(\gamma_1, \gamma_2)$, $(\gamma_2,\gamma_3)$ and $(\gamma_1,\gamma_3)$. Suppose that $\alpha_{12} \leq \alpha_{23} \leq \alpha_{13}$. Then $\alpha_{12}=\alpha_{23}$.}
\end{remark}
\begin{proof}
$\|\gamma_1(t)-\gamma_2(t)\|\leq \|\gamma_1(t)-\gamma_3(t)\|+\|\gamma_3(t)-\gamma_2(t)\| \Rightarrow \alpha_{12} \geq \alpha_{23}$.
\end{proof}
Observe that the function $d_{P}(\gamma_1(t), \gamma_2)$ given by 
\begin{equation}\label{distance}
d_{P}(\gamma_1(t), \gamma_2)=inf\{d_P(\gamma_1(t),y); y \in \gamma_2\}
\end{equation}
is a semialgebraic function and $d_{P}(\gamma_1(0), \gamma_2)=0$. Then, $ord_t d_{P}(\gamma_1(t), \gamma_2)$ is well defined.

\begin{lemma} (Inner order comparison lemma) \   
$tord_{inn}(\gamma_1,\gamma_2)=ord_t (d_{P}(\gamma_1(t), \gamma_2))$.
\end{lemma}
\begin{proof}

Consider the pancake decomposition $\{X_i\}$ of $X$, where the metric $d_P$ corresponds to this decomposition (see \cite{BM}). By definition of metric $d_P$, we choose semialgebraic arcs $\tilde{\beta}_1, \dots \tilde{\beta}_N$, $\tilde{\beta}_i(0)=x_0$, $i=1,\ldots,N$, such that all the pairs $(\gamma_1,\tilde{\beta}_1), (\tilde{\beta}_1,\tilde{\beta}_2),\ldots, (\tilde{\beta}_N,\gamma_2)$ belong to the same ``pancake", i.e., a normally embedded subset of $X$ (see \cite{BM}). For $t \in [0,\delta), \ \delta$ sufficient small, we have
\begin{center}
$d_P(\gamma_1(t),\gamma_2)= \|\gamma_1(t)-\tilde{\beta}_1(t)\|+\|\tilde{\beta}_1(t)-\tilde{\beta}_2(t)\|+\ldots+\|\tilde{\beta}_N(t)-\tilde{\gamma}_2(t)\|$.
\end{center}
Notice that arcs $\tilde{\beta}_i$ are not necessarily parametrized by the distance to $x_0$. Moreover, we have
\begin{equation}
\|\gamma_1(t)-\tilde{\beta}_S(t)\| \geq d_{outer}(\gamma_1(t), \tilde{\beta}_S), \forall S, 
\label{ine1}
\end{equation}
 where $d_{outer}(\gamma_1(t), \tilde{\beta}_S)$ is defined as in (\ref{distance}), considering the euclidean distance. Otherwise, 
\begin{center}
$d_P(\gamma_1(t),\gamma_2(t)) \leq \|\gamma_1(t)-\beta_1(t)\|+\|\beta_1(t)-\beta_2(t)\|+\ldots+\|\beta_N(t)-\gamma_2(t)\|,$
\end{center}
where now $\beta_i(t)$ and $\gamma_2(t)$ is a parametrization of $\tilde{\beta_i}$ and $\gamma_2$ by distance to the origin. By the non-archimedean property, we have
\begin{center}
$\|\gamma_1(t)-\beta_1(t)\|+\|\beta_1(t)-\beta_2(t)\|+\ldots+\|\beta_N(t)-\gamma_2(t)\| \simeq \|\beta_S(t)-\beta_{S-1}(t)\|,$
\end{center}
for some $S \in \{1,\ldots,N+1\}$, where $\beta_0(t)$ is $\gamma_1(t)$ and $\beta_{N+1}(t)$ is $\gamma_2(t)$ and
\begin{center}
$\|\beta_S(t)-\beta_{S-1}(t)\|\simeq \|\gamma_1(t)-\beta_S(t)\|$,
\end{center}
where $f(t)\simeq g(t)$ means that the functions have the same order. Now, the outer order comparison lemma of \cite{BF} says:
\begin{center}
$\|\gamma_1(t)-\beta_S(t)\|\simeq d_{outer}(\gamma_1(t),\tilde{\beta}_S)$,
\end{center}
so, there exists constant $C_2>0$ such that 
\begin{equation}
d_P(\gamma_1(t),\gamma_2(t))\leq C_2d_{outer}(\gamma_1(t),\tilde{\beta}_S).
\label{ine2}
\end{equation}
Hence, by (\ref{ine1}) and (\ref{ine2}) the lemma is proved.
\end{proof}
End of the proof of Theorem 2.2: 

Since $d_X$ is bi-Lipschitz equivalent to the $d_P$, using the previous lemma and the inequality (\ref{order condition}), we obtain that
\begin{center}
$\lim_{t \to 0^+}\frac{d_X(\gamma_1(t),\gamma_2(t))}{\|\gamma_1(t)-\gamma_2(t)\|}=+\infty$,
\end{center}
or, in other words, $tord(\gamma_1,\gamma_2) > tord_{inn}(\gamma_1,\gamma_2)$, where $\gamma_1, \ \gamma_2$ can be considered parametrized by the distance to the point $x_0$. 
\end{proof}

\end{document}